%% file: main.tex
\documentclass[letterpaper]{amsart}
\input{./header.tex}
\title[Explicit Riem.\ mfds with unexpectedly behaving center of mass]{Explicit Riemannian manifolds with\\unexpectedly behaving center of mass}%
 \author{Carla Cederbaum}
 \address{Mathematisches Institut\\Universit\"at T\"ubingen\\Auf der Morgenstelle 10\\72076 T\"ubingen, Germany}%
 \email{cederbaum@math.uni-tuebingen.de}
 \author{Christopher Nerz}
 \address{Mathematisches Institut\\Universit\"at T\"ubingen\\Auf der Morgenstelle 10\\72076 T\"ubingen, Germany}%
 \email{Christopher.Nerz@phoenixes.de}
 \date\today
 \thanks{This material is based upon work supported by the National Science Foundation under Grant No.\,0932078 000, while the authors were in residence at the Mathematical Science Research Institute in Berkeley, California, during the fall semester 2013.}
\ifpdf\hypersetup{pdftitle= {Explicit Riemannian manifolds with unexpectedly behaving center of mass},
									pdfauthor		= {Carla Cederbaum, Christopher Nerz},
									pdfsubject	= {Construction and discussion of explicit examples of Riemannian manifolds without well-defined center of mass},
									pdfkeywords	= {explicit example, constant mean curvature foliation, ADM center of mass, CMC center of mass, center of mass}}\fi
\begin{document}
\begin{abstract}
The (relativistic) center of mass of an asymptotically flat Riemannian manifold is often defined by certain surface integral expressions evaluated along a foliation of the manifold near infinity, e.\,g.\ by Arnowitt, Deser, and Misner (ADM). There are also what we call \emph{abstract} definitions of the center of mass in terms of a foliation near infinity itself, going back to the constant mean curvature (CMC-) foliation studied by Huisken and Yau; these give rise to surface integral expressions when equipped with suitable systems of coordinates. We discuss subtle asymptotic convergence issues regarding the ADM- and the coordinate expressions related to the CMC-center of mass. In particular, we give explicit examples demonstrating that both can diverge -- in a setting where Einstein's equation is satisfied. We also give explicit examples of the same asymptotic order of decay with prescribed mass and center of mass. We illustrate both phenomena by providing analogous examples in Newtonian gravity. Our examples conflict with some results in the literature. 
\end{abstract}
\maketitle%

\input{introduction}
\input{notation}
\input{Newtonian}
\input{general}
\input{York_Ex}
\input{Schwarzschild}
\input{prescribed}
\input{table}
\clearpage\bibliography{./bib}\bibliographystyle{alpha}
\pagebreak[1]$\,\!$\nopagebreak\vfill
\end{document}

%% file: header.tex
\usepackage{ifpdf}
\ifpdf
 \usepackage[implicit,naturalnames,hyperfootnotes=false,
						final]{hyperref}																				%
 \usepackage[pdftex]{color}																					
\else
 \usepackage{color}																									
 \gdef\texorpdfstring#1#2{#1}
\fi
\usepackage[refpage]{nomencl}																				
\usepackage{lmodern}																								
\usepackage{xcolor}																									
\usepackage{array}	
\usepackage{multirow}																									%
\usepackage{amsmath}																								
\usepackage{amssymb}																								
\usepackage{amsthm}																									
\usepackage{bbm} 																	
\usepackage{amsfonts}																								
\usepackage{mathrsfs}																								
\usepackage{calligra}																								
\usepackage{dsfont}																									
\usepackage{nicefrac}																								
\usepackage{esint}																									
\usepackage{mathtools}																							
\usepackage{ifthen}																									
\usepackage{xparse}																									
\usepackage{calc}																										
\usepackage{partial}																
\newtheorem{theorem}{Theorem}[section]
\newtheorem{corollary}[theorem]{Corollary}
\newtheorem{proposition}[theorem]{Proposition}
\theoremstyle{definition}
\newtheorem{definition}[theorem]{Definition}
\newtheorem{example}[theorem]{Example}
\newtheorem{notation}[theorem]{Notation}
\theoremstyle{remark}
\newtheorem{remark}[theorem]{Remark}
\newtheorem{remark*}[theorem]{Remark}
\NewDocumentCommand\term{so}{\IfValueTF{#2}\term\emph}
\gdef\Oof{\DoOof{\mathcal O}}
\NewDocumentCommand\DoOof{ms}{\IfBooleanTF{#2}{#1}{\DoOofsnd{#1}}}
\NewDocumentCommand\DoOofsnd{mt_}
{\IfBooleanTF{#2}{\DoOofthd{#1_}}{\DoOofthd{#1}\relax}}
\NewDocumentCommand\DoOofthd{mmO{r}t^t-t+m}
{#1{#2}(#3^{\IfBooleanTF{#5}-\relax#7})}
\gdef\erA{B}%
\gdef\erAt{\DoATensor[\hspace{-.075em}](t){B}}%
\gdef\schwg{\g[m]}
\gdef\eukg{\delta}
\gdef\PiT{\DoATensor\Pi}
\NewDocumentCommand\linmomentum{d<>ot_t^}
{\IfBooleanTF{#3}{(\DoATensor{P_{\text{ADM}}}<#1>[#2])_}
 {\IfBooleanTF{#4}{(\DoATensor{P_{\text{ADM}}}<#1>[#2])^}{\DoATensor{\vec P_{\text{ADM}}}<#1>[#2]}}}
\let\g\metric
\gdef\schwN{\DoATensor[\hspace{-.075em}](m){N}}
\gdef\Tex{\DoATensor[\hspace{+0.05em}](m,\vec{z}){T}}
\gdef\lambdaex{\DoATensor[\hspace{-.05em}](m,\vec{z}){\lambda}}
\gdef\Tnormal{\DoATensor[\hspace{-.05em}](T){n}}
\gdef\slice{\DoATensor[\hspace{-.15em}](T){M^{3}}}
\gdef\devi{\DoATensor[\hspace{-.1em}](T){B}}					%
\gdef\K{\DoATensor[\hspace{-.1em}](T){K}}					%
\gdef\Ht{\DoATensor[\hspace{-.1em}](T){H}}					%
\gdef\scalar{\DoATensor{R}}
\let\sc\scalar
\let\k\II
\gdef\kY{\DoATensor[\!]{\textrm Y}}
\gdef\normal{\DoATensor[\!]{\nu}}
\gdef\laplace{\DoATensor\Delta}
\gdef\div{\DoATensor[\hspace{-.05em}]{\text{div}}}
\gdef\tr{\DoATensor[\hspace{-.05em}]{\text{tr}}}
\gdef\vol{\DoATensor{\text{V}}}
\gdef\Ag{\DoATensor<\hspace{-.05em}>A}
\gdef\momentum\Pi
\gdef\momenden{J}
\NewDocumentCommand\M{O{3}t^}{\IfBooleanTF{#2}{M^}{M^{#1}}}
\gdef\Mz{\DoATensor<\hspace{-.05em}>\Sigma}
\NewDocumentCommand\centerz{d<>ost_t^}
{\IfBooleanTF{#4}{\centerzsnd<#1>[#2]_}
 {\IfBooleanTF{#5}{\centerzsnd<#1>[#2]^}
  {\DoATensor<\hspace{-.05em}>{\vec z}<#1>[#2]}}}
\NewDocumentCommand\centerzsnd{d<>omm}
{\ifthenelse{\equal{#4}{\text{ADM}}}
 {\DoATensor<\hspace{-.05em}>{\vec z_{\text{ADM}}}<#1>[#2]}
 {\ifthenelse{\equal{#4}{ADM}}
  {\DoATensor<\hspace{-.05em}>{\vec z_{\text{ADM}}}<#1>[#2]}
  {\ifthenelse{\equal{#4}{\text{CMC}}}
   {\DoATensor<\hspace{-.05em}>{\vec z_{\text{CMC}}}<#1>[#2]}
   {\ifthenelse{\equal{#4}{CMC}}
    {\DoATensor<\hspace{-.05em}>{\vec z_{\text{CMC}}}<#1>[#2]}
    {\ifthenelse{\equal{#4}{\text{N}}}
     {\DoATensor<\hspace{-.05em}>{\vec z_{\text N}}<#1>[#2]}
     {\ifthenelse{\equal{#4}{N}}
       {\DoATensor<\hspace{-.05em}>{\vec z_{\text N}}<#1>[#2]}
       {\ifthenelse{\equal{#4}{\text{E}}}
        {\DoATensor<\hspace{-.05em}>{\vec z_{E}}<#1>[#2]}
        {\ifthenelse{\equal{#4}{E}}
         {\DoATensor<\hspace{-.05em}>{\vec z_{\text E}}<#1>[#2]}
         {{\DoATensor<\hspace{-.05em}>{z}<#1>[#2]#3{#4}}}}}}}}}}}
\undef\d
\NewDocumentCommand\d{s}{\IfBooleanTF{#1}\relax{\mathop{}\!}\mathrm d}%
\newcommand\R{\mathds{R}}																						
\newcommand\Ck{\mathcal C}
\newcommand\Hk{\textrm H}
\newcommand\sphere{\mathbb S}																				
\newcommand\ve{\varepsilon}																					
\usepackage{urwchancal}																							
\DeclareMathAlphabet{\mathcal}{OT1}{pzc}{m}{n}
\mathchardef\ordinarycolon\mathcode`\:
\mathchardef\ordinaryequal\mathcode`\=
\NewDocumentCommand\mathcolon{t=}
{\IfBooleanTF{#1}\coloneqq\ordinarycolon}
\NewDocumentCommand\mathequal{t:}
{\IfBooleanTF{#1}\eqqcolon\ordinaryequal}
\AtBeginDocument{\mathcode`\==\string"8000}
\AtBeginDocument{\mathcode`\:=\string"8000}
\begingroup \catcode`\:=\active \gdef:{\mathcolon} \endgroup
\begingroup \catcode`\==\active \gdef={\mathequal} \endgroup
\def\thesupercolon
{\mathrel{\vcenter{\baselineskip0.5ex \lineskiplimit0pt
\hbox{\small.}\hbox{\small.}}}}
\gdef\coloneqq{\thesupercolon\ordinaryequal}
\gdef\eqqcolon{\ordinaryequal\thesupercolon}
\usepackage{ifdraft}																								
\ifdraft{
\def\hyphenateWholeString #1{\xHyphenate#1$\wholeString\unskip}			
\def\xHyphenate#1#2\wholeString {\if#1$
\else\transform{#1}
\takeTheRest#2\ofTheString\fi}
\def\takeTheRest#1\ofTheString\fi
{\fi \xHyphenate#1\wholeString}
\def\transform#1{#1\hskip 0pt plus 1pt}
\makeatletter																												%
\usepackage[notref]{showkeys}																				
\renewcommand*\showkeyslabelformat[1]{
 \let\raggedleftmarginnoteold\raggedleftmarginnote
 \let\raggedrightmarginnoteold\raggedrightmarginnote
 \renewcommand*{\raggedleftmarginnote}
	{\raggedleftmarginnoteold\centering\tiny\linespread0\selectfont}
 \renewcommand*{\raggedrightmarginnote}
	{\raggedrightmarginnoteold\centering\tiny\linespread0\selectfont}
 \marginnote[\vspace{-2.25em}
 \fparcolorbox\correctionboxlength{black}{white}
 {\noindent\hyphenateWholeString{#1}}]{}
 \let\raggedleftmarginnote\raggedleftmarginnoteold
 \let\raggedrightmarginnote\raggedrightmarginnoteold}
\makeatother																												%
}{}																																	
\usepackage{auto-eqlabel}																						
\makeatletter
\NewDocumentCommand\DoATensor{sD<>{}O{}d<>d()md<>os}
{\IfBooleanTF{#1}{#6}
 {\IfValueTF{#4}{\IfValueTF{#7}{\DoATensor<#2>[#3](#5){#6}<#7>[#8]}
											{\DoATensor<#2>[#3](#5){#6}<#4>[#8]}}
  {\IfValueTF{#5}{\IfValueTF{#8}{\DoATensor<#2>[#3]{#6}<#7>[#8]}
											{\DoATensor<#2>[#3]{#6}<#4>[#5]}}
   {\DoATensor@start{#6}
    {\IfValueTF{#7}
     {\IfValueTF{#8}{\TensorIndex*[#7#2]<#8#3>{#6}}
	    {\TensorIndex*[#7#2]{#6}}}
	   {\IfValueTF{#8}{\TensorIndex*<#8#3>{#6}}
	    {\TensorIndex*{#6}}}}}}}}
\NewDocumentCommand\DoATensor@start{smm}
{\DoATensor@checkforsubindex{#1}{#2}{#3}}
\NewDocumentCommand\DoATensor@checkforsubindex{mmmt_}
{\IfBooleanTF{#4}{\DoATensor@index{#2}{#3}_}
 {\DoATensor@checkforsupindex{#1}{#2}{#3}}}
\NewDocumentCommand\DoATensor@checkforsupindex{mmmt^}
{\IfBooleanTF{#4}{\DoATensor@index{#2}{#3}^}
 {\IfBooleanTF{#1}{\mathop{#3}}{#3}}}
\NewDocumentCommand\DoATensor@index{mmmm}
{\DoATensor@start*{#1}{#2\vphantom{#1}#3{#4}}}
\NewDocumentCommand\TensorIndex{ssod<>msod<>}
{\IfBooleanTF{#1}
 {\TensorIndex[#3]<#4>{#5}*[#7]<#8>}
 {\IfBooleanTF{#2}
  {\mathop{\TensorIndex[#3]<#4>{#5}*[#7]<#8>}}
  {\IfBooleanTF{#6}
   {\Tensor@Indexparse@@left{\IfValueTF{#3}{_{#3}}\relax}{\IfValueTF{#4}{^{#4}}\relax}{#5}%
    #5
    \IfValueTF{#7}{_#7}\relax\IfValueTF{#8}{^#8}\relax}
   {\IfValueTF{#7}{\mathop{\TensorIndex[#3]<#4>{#5}*[#7]<#8>}}
    {\IfValueTF{#8}{\mathop{\TensorIndex[#3]<#4>{#5}*[#7]<#8>}}
	   {\IfValueTF{#3}{\mathop{}\TensorIndex[#3]<#4>{#5}*[#7]<#8>}
	    {\IfValueTF{#4}{\mathop{}\TensorIndex[#3]<#4>{#5}*[#7]<#8>}{\mathop{#5}}}}}}}}}
\newdimen\Tensor@widthmax
\newdimen\Tensor@widthfirst
\newdimen\Tensor@widthsecond
\NewDocumentCommand\Tensor@Indexparse@@left{smmm}
{\IfBooleanTF{#1}
 {\setlength\Tensor@widthfirst{\widthof{\ensuremath{\vphantom{#4}#2}}}%
  \setlength\Tensor@widthsecond{\widthof{\ensuremath{\vphantom{#4}#3}}}%
  \setlength\Tensor@widthmax{\maxof\Tensor@widthfirst\Tensor@widthsecond}
  \kern\Tensor@widthmax
	\kern-\Tensor@widthfirst%
	\vphantom{#4}#2
	\kern-\Tensor@widthsecond%
	\vphantom{#4}#3
  \kern-\scriptspace}
 {\ifthenelse{\equal{#2}\relax}
  {\ifthenelse{\equal{#3}\relax}\relax
	 {\Tensor@Indexparse@@left*{#2}{#3}{#4}}}
	{\Tensor@Indexparse@@left*{#2}{#3}{#4}}}}
\makeatother
\NewDocumentCommand\sinterval{smsm}
{\underline{\IfBooleanTF{#1}[\({{#2}};{{#4}}\IfBooleanTF{#3}]\)}}
\NewDocumentCommand\interval{st-mst-m}
{\IfBooleanTF{#1}{\leftonlyinbigmath[}(
 \IfBooleanTF{#2}{{-}}\relax#3,\IfBooleanTF{#5}{{-}}\relax#6
 \IfBooleanTF{#4}{\rightonlyinbigmath]})}
\newcommand\leftonlyinbigmath{\ifbracketer\goodleft\relax}
\newcommand\rightonlyinbigmath{\ifbracketer\goodright\relax}
\gdef\{{\oldlbrace}
\gdef\}{\oldrbrace}
\let\lbracket(
\let\rbracket)
\gdef\({\mathchar'50}
\gdef\){\mathchar'51}
\let\oldlbrace\lbrace
\let\oldrbrace\rbrace
\let\oldvert\vert
\let\oldVert\Vert
\let\oldmath\[
\let\oldmathend\]
\def\ismath{}
\renewcommand\[{\oldmath\def\ismath{1}}
\renewcommand\]{\oldmathend\def\ismath{}}
\let\oldcases\cases
\def\cases{\let\tmp\lbrace
 \def\lbrace{\{}\oldcases
 \let\lbrace\tmp}																										%
\makeatletter
\gdef\goodleft{\mathopen{}\mathclose\bgroup\left}
\gdef\goodright{\aftergroup\egroup\right}
\@ifundefined{resetMathstrut@}\relax{
 \global\def\resetMathstrut@{
  \setbox\z@\hbox{
   \mathchardef\@tempa\mathcode`\[\relax														%
   \def\@tempb##1"##2##3{\the\textfont"##3\char"}
   \expandafter\@tempb\meaning\@tempa \relax												%
  }
 \ht\Mathstrutbox@\ht\z@ \dp\Mathstrutbox@\dp\z@}}
\begingroup\lccode`~=`(
 \lowercase{\endgroup\def~{\ifbracketer{\goodleft(}\(}}
 \mathcode`(="8000
\begingroup\lccode`~=`)
 \lowercase{\endgroup\def~{\ifbracketer{\goodright)}\)}}
 \mathcode`)="8000
 \newcommand\ifbracketer[2]{\if@display #1 \else \ifthenelse{\equal{\ismath}{}}{#2}{#1} \fi}%
 \renewcommand\lvert{\ifbracketer{\goodleft\oldvert\renewcommand\vert\rvert}\oldvert}%
 \renewcommand\rvert{\ifbracketer{\goodright\oldvert\renewcommand\vert\lvert}\oldvert}%
 \renewcommand\vert{\ifbracketer\lvert\oldvert}%
 \renewcommand\lVert{\ifbracketer{\goodleft\oldVert\renewcommand\Vert\rVert}\oldVert}%
 \renewcommand\rVert{\ifbracketer{\goodright\oldVert\renewcommand\Vert\lVert}\oldVert}%
 \renewcommand\Vert{\ifbracketer\lVert\oldVert}%
 \renewcommand\lbrace{\ifbracketer{\goodleft\oldlbrace\renewcommand\brace\rbrace}\oldlbrace}
 \renewcommand\rbrace{\ifbracketer{\goodright\oldrbrace\renewcommand\brace\lbrace}\oldrbrace}
 \newcommand\BVert{\ifbracketer\lBVert{\oldVert\hspace{-.1em}\oldvert}}%
 \newcommand\lBVert{\ifbracketer{\goodleft\oldVert\hspace{-.1em}\goodleft\oldvert\renewcommand\BVert\rBVert}\BVert}
 \newcommand\rBVert{\ifbracketer{\goodright\oldVert\hspace{-.1em}\goodright\oldvert\renewcommand\BVert\lBVert}\BVert}
\makeatother
\newcommand\phantomas[3][c]{
	\ifmmode
		\makebox[\widthof{$#2$}][#1]{$#3$}
	\else
		\makebox[\widthof{#2}][#1]{#3}
	\fi
}
\let\oldint\int\gdef\int{\oldint\limits}
\let\oldfint\fint\gdef\fint{\oldfint\limits}

%% file: introduction.tex
 \section{Introduction and general considerations}\label{intro}
Isolated gravitating systems are individual or clusters of stars, black holes, or galaxies that do not interact with any matter or gravitational radiation outside the system under consideration. Intuitively, they should have a (total) mass and a (total) center of mass. We will demonstrate how delicately different notions of (total) center of mass described in the literature depend on the decay of the matter near 'infinity', i.\,e.\ far away from the main constituents of the system.

Let us begin by stating (somewhat vague) definitions of isolated gravitating systems to be made precise in Sections \ref{NG} and \ref{general}. Isolated gravitating systems have been studied intensively both in Newtonian gravity (NG) and general relativity (GR). In Newtonian gravity, an isolated gravitating system is described by a matter density function $\rho:\R^{3}\to\left[0,\infty\right)$ that decays suitably fast at infinity. It is useful to introduce the Newtonian potential $U$ solving the Poisson equation
\begin{align*} \labeleq{NPoisson}
\laplace U ={}& 4\pi\rho\quad\text{ in }\R^3, \\ \labeleq{NPoissonbdry}
 U \to{}& 0 \qquad\text{ as }r\to\infty,
 \end{align*}
where we have set the gravitational constant $G$ to $1$. $U$ inherits its decay from $\rho$. 

In general relativity, isolated gravitating systems are modeled as asymptotically flat $3$-dimensional Riemannian manifolds $({\M^3},\g)$ equipped with a symmetric $(0,2)$-tensor field $K$ suitably decaying at infinity in the asymptotic end of the manifold. In addition, an isolated gravitating system carries a suitably decaying mass density function $\rho:\M^3\to\R$ as well as a suitably decaying momentum density $1$-form $J$. The matter variables $\rho$ and $J$ are usually assumed to satisfy the \emph{dominant energy condition}
\begin{align*}\labeleq{DEC}
\rho\geq\vert J\vert
\end{align*}
on $\M^3$, where $\vert\cdot\vert$ denotes the tensor norm induced by $\g$. $({\M^3},\g)$ can be pictured as a Riemannian submanifold or \emph{time-slice} of a Lorentzian spacetime $(\mathfrak{M}^{4},\mathfrak{g})$. The tensor field $K$ then arises as the induced second fundamental form. In this view, $\rho$ and $J$ arise as the time-time and the time-space components of the energy-momentum tensor $\mathfrak T$ of general relativity, respectively. The relativistic spacetime $(\mathfrak{M}^{4},\mathfrak{g})$ must satisfy Einstein's equation
\begin{align*}\labeleq{EE}
\mathfrak{Ric}-\frac{1}{2}\,\mathfrak{R}\,\mathfrak{g}={8\pi}\mathfrak T,
\end{align*}
where we have set both the gravitational constant $G$ and speed of light $c$ to $1$. $\mathfrak{Ric}$ and $\mathfrak{R}$ denote the Ricci and scalar curvature of the Lorentzian metric $\mathfrak{g}$, respectively. Applying the Gau{\ss}-Mainardi-Codazzi equations to $({\M^3},\g)\hookrightarrow(\mathfrak{M}^{4},\mathfrak{g})$, Einstein's equation \eqref{EE} implies the \emph{Einstein constraint equations}
\begin{align*}
\scalar-\vert K\vert^2+(\tr K)^{2}={}&16\pi\rho \labeleq{Econstraint}\\
\div(K-\tr K\,\g)={}&{}8\pi J. \labeleq{Mconstraint}
\end{align*}
Here, $\tr$ and $\div$ denote the trace and the divergence with respect to $\g$, respectively. Collecting all the above pieces of information, we say that an isolated gravitating system in GR is a so-called \emph{initial data set} $({\M^3},\g,K,\rho,J)$ of suitable decay and satisfying \eqref{Econstraint} and \eqref{Mconstraint}. However, many properties of initial data sets $({\M^3},\g,K,\rho,J)$ such as e.\,g.\ mass and center of mass can be defined using only the Riemannian variables $({\M^3},\g)$ and we will frequently switch back and forth between the initial data and the Riemannian perspectives. Also, we will explicitly name and denote the asymptotic coordinate chart whenever it seems helpful.

This paper is concerned with delicate decay considerations related to the (total) center of mass (CoM) of isolated gravitating systems in GR and NG, the latter being regarded as an indicator of phenomena and an incubator for ideas. At the same time, we will illustrate some more or less well-known analogies between the notions of center of mass in NG and GR, see Figure \ref{table}. When considering the center of mass, we will always implicitly or explicitly assume that the related mass is strictly positive. This is both a technical (as many definitions of center of mass divide by the mass) and a physically reasonable assumption.

\subsection{Mass in NG and GR}
In NG, the mass and the center of mass of an isolated gravitating system are conventionally determined by volume integrals of the matter density $\rho$ and the weighted position vector $\rho\,\vec{x}$ over the entire space $\R^{3}$, respectively, see \eqref{Nmass} and \eqref{NCoM}. The question whether the center of mass is well-defined thus depends on the integrability properties of $\rho\,\vec{x}$ in the sense of indefinite Riemann integrals (in spherical polars). On the other hand, as $\rho$ is essentially a divergence by \eqref{NPoisson}, the volume integrals for mass and center of mass can be rewritten in terms of surface integrals \lq at infinity\rq, see Section \ref{NG}.

When an isolated gravitating system is modeled as an initial data set in GR, its (total) relativistic mass, the so-called \emph{ADM-mass} introduced by Arnowitt, Deser, and Misner \cite{arnowitt1961coordinate}, is given by a surface integral expression \lq at infinity\rq, see \eqref{mADM}. The ADM-mass is known to asymptotically correspond to the volume integral over the scalar curvature $\sc$ in a sense made precise e.\,g.\ by Bartnik \cite{Bartnik}. Namely, for any suitably asymptotically flat Riemannian manifold $(\M,\g)$, the ADM-mass $m_{\text{ADM}}$ converges if and only if the volume integral
\begin{align*} \int \sc \d\vol \end{align*}
is finite (when taken over a suitable neighborhood of infinity). In particular, the ADM-mass is independent of the particular asymptotic coordinate chart used to compute the ADM- or the associated volume integral. The volume integral $\int\sc\d V$ is conceptually related to the Newtonian mass volume integral $\int\rho\d V$ via \eqref{Econstraint}.

\subsection{Center of mass in GR}
We now proceed to discussing the (total) center of mass of isolated gravitating systems in GR. First, let us remark that the contemporary literature knows several definitions of such a center. Several authors define the center of mass of an initial data set $(M^{3},g,K,\rho,J)$ as a foliation near infinity of the corresponding Riemannian manifold $(M^{3},\g)$ with total mass $m>0$. We will call such definitions \term{abstract} to contrast what we will call \term*[center of mass!coordinate]{coordinate} definitions of center of mass, see below.

The first such abstract definition was given by Huisken and Yau \cite{huisken_yau_foliation}, who defined the \emph{CMC-center of mass} to be the unique foliation near infinity by closed, stable surfaces with constant mean curvature, see Section \ref{general}. This was motivated by an idea of Christodoulou and Yau \cite{christodoulou71some}. Lamm, Metzger, and Schulze \cite{lamm2011foliationsbywillmore} sub\-sequent\-ly used a unique foliation by spheres of Willmore type. In the case of static isolated gravitational systems with compactly supported matter, the first author used level sets of the static lapse function to define an abstract center of mass  \cite[Chap.~3,\;5]{Cederbaum_newtonian_limit}. We will interpret the foliation near infinity by level sets of the Newtonian potential as an abstract center of mass in NG in the same spirit, see Section \ref{NG} and Figure \ref{table}.

Chen, Wang, and Yau recently suggested a new definition of (quasi-local and total) center of mass of isolated gravitating systems which is constructed from optimal isometric embeddings into Minkowski spacetime \cite[Def.~3.2]{chen_wang_yau__Quasilocal_angular_momentum_and_center_of_mass_in_gr}.

Other definitions of relativistic CoM assign a specific center of mass vector $\vec{z}\in\R^{3}$ to an isolated gravitating system. The vector $\vec{z}$ can be pictured to describe a point in the target $\R^3$ of the asymptotically flat coordinate chart. The specific vector $\vec{z}$ in the various constructions depends on the choice of coordinates near infinity (i.\,e.\ in the asymptotically flat end of the manifold) -- at least a priori. In particular, the CoM vectors $\vec{z}$ transform appropriately under Euclidean motions applied to the chosen asymptotic coordinates. To distinguish such definitions from the abstract definitions described above, we will call them \term*[center of mass!coordinate]{coordinate} centers of mass.

The most important coordinate centers of mass in view of this paper are the \emph{(coordinate) ADM-} and the \emph{coordinate CMC-center of mass}. The ADM-center goes back to Arnowitt, Deser, and Misner \cite{arnowitt1961coordinate} and Beig and {\`O} Murchadha \cite{BM_examples} and is constructed in analogy to the ADM-mass, see Definition \ref{ADMCoM}. The coordinate CMC-center of mass is constructed from the abstract CMC-center and was suggested by Huisken and Yau \cite{huisken_yau_foliation} as a Euclidean center of the CMC-foliation, see Definition \ref{folicoordCoM}. Other coordinate CoM notions have been suggested by Corvino and Schoen \cite{corvino2006asymptotics} and by Huang \cite{Lan_Hsuan_Huang__Foliations_by_Stable_Spheres_with_Constant_Mean_Curvature}.

It is well-known that the coordinate CMC-center of mass of a given, suitably asymptotically flat Riemannian manifold $(M^{3},\g,\vec{x}\,)$ coincides with its ADM-center of mass. This has been established by Huang \cite{Lan_Hsuan_Huang__Foliations_by_Stable_Spheres_with_Constant_Mean_Curvature}, Eichmair and Metzger \cite{metzger_eichmair_2012_unique}, and by the second author \cite[Cor.~3.8]{nerz2013timeevolutionofCMC} under different assumptions on the asymptotic decay, see Section \ref{general}.

In the context of static isolated gravitating systems with compactly supported matter, the first author defined a \lq pseudo-Newtonian\rq\ (quasi-local and total) coordinate center of mass and proved that it coincides with the coordinate CMC- and ADM-centers of mass (in appropriate geometric asymptotic coordinates). The coordinate pseudo-Newtonian and thus the coordinate CMC- and ADM-centers of mass converge to the Newtonian one in the Newtonian limit $c\to\infty$ \cite[Chap.~4,\;6]{Cederbaum_newtonian_limit}. This further justifies the ADM- and CMC-definitions of center of mass.

The ADM-center of mass expression is known to asymptotically correspond to the volume integral of the scalar curvature $\sc$ integrated against $\vec{x}$, i.\,e.\ for any $\Ck^2$-asymptotically Schwarzschildean Riemannian manifold $(\M,\g,\vec{x}\,)$, the ADM-center of mass $\centerz_{\text{ADM}}$ converges if and only if the volume integral
\begin{align*} \int \sc\,\vec x \d\vol \end{align*}
is finite (when taken over a suitable neighborhood of infinty), see Corvino and Wu \cite{Corvino-Wu}, Definition \ref{Ck_asymptotically_Schwarzschildean}, and page \pageref{laplaHess}. As for the ADM-mass, the volume integral $\int\sc\,\vec{x}\d V$ is conceptually related to the Newtonian integral $\int\rho\,\vec{x}\d V$ via \eqref{Econstraint}.

\subsection{Critical order of decay}
The Newtonian and relativistic center of mass volume integrals $\int\rho\,\vec{x}\d V$ (NG) and $\int \sc\,\vec{x}\d V$ (GR) clearly converge whenever the matter density $\rho$ and the scalar curvature $\sc$ decay faster than $r^{-4}$ as $r\to\infty$ (by Lebesgue's theorem on dominated convergence). If they decay more slowly, in particular exactly at the \emph{critical rate $r^{-4}$}, the center of mass volume integrals will only converge under certain additional conditions. For example, the Newtonian integral converges (in the sense of indefinite Riemann integrals in spherical polars) if $\rho$ is even or asymptotically even as defined on page \pageref{antisymm}. Accordingly, the relativistic integral converges (in the same sense) if the Riemannian metric $\g$ and in particular its scalar curvature $\sc$ satisfy the asymptotic evenness conditions specified by Regge and Teitelboim \cite{RT}. This was proved by Huang \cite{Lan_Hsuan_Huang__Foliations_by_Stable_Spheres_with_Constant_Mean_Curvature} (in dimension $n=3$) and Eichmair and Metzger \cite{metzger_eichmair_2012_unique} (for $n\geq3$).

At the level of the Newtonian potential $U$, a critically decaying matter density $\rho=\mathcal{O}(r^{-4})$ of mass $m>0$ gives rise to an asymptotic expansion of the form
\begin{align*}\labeleq{Uexp}
U=-\frac{m}{r}+\mathcal{O}_{2}(r^{-2})
\end{align*}
by classical potential theory. The term $-\frac{m}{r}$ is harmonic and arises as the potential of a spherically symmetric (or point) source of mass $m$ centered at the origin $\vec{0}$. In the relativistic setting, a corresponding natural asymptotic decay assumption on a Riemannian metric $\g$ is to be asymptotic to the spherically symmetric Riemannian Schwarzschild metric ${\g[m]}$ of mass $m>0$ in a given system of asymptotic coordinates, see Definition \ref{Ck_asymptotically_Schwarzschildean}. The Riemannian Schwarzschild metric $\g[m]$ has vanishing scalar curvature ${\sc[m]}\equiv0$. A metric that deviates from the Schwarzschild metric at the order $r^{-2}$ can thus give rise to a scalar curvature function $\sc$ decaying exactly at the critical rate $r^{-4}$ (if the contribution of the deviation to the scalar curvature does not cancel at highest order). The examples we will present in Sections \ref{NG} through \ref{prescribed} all decay exactly at the critical rate, see also Figure \ref{table}.
\pagebreak[3]

\subsection{Systems with divergent CoM} 
In this paper, we present explicit examples of asymptotically Schwarzschildean Riemannian manifolds (or initial data sets) of critical decay for which the coordinate CMC- and the ADM-center of mass do not converge. The decay assumptions we make fall within the categories of metrics studied by Huisken and Yau \cite{huisken_yau_foliation} and by Metzger \cite{metzger2007foliations} who prove existence of the abstract CMC-center of mass under stronger and weaker asymptotic conditions, respectively. The examples do not satisfy the Regge-Teitelboim conditions. While the Euclidean centers of the leaves of the CMC-foliation near infinity stay bounded as asserted in \cite[Proof of Thm~5.1]{huisken_yau_foliation}, the divergent examples conflict with the convergence to a coordinate CMC-center of mass stated in \cite[Thm~4.2]{huisken_yau_foliation}.

The first class of examples is derived from an understanding of the time evolution of the leaves of the CMC-foliation established by the second author in \cite[Thm~3.1]{nerz2013timeevolutionofCMC}. Here, we modify the Riemannian Schwarzschild metric by a term of the form $-2 f(r)\,Y=\mathcal{O}(r^{-2})$, where $f$ is a suitably chosen function and $Y$ is the York tensor introduced by York \cite{York__1979__Kinematics_and_dynamics_of_general_relativity} when considering the ADM-linear momentum. Relying on the techniques developed in \cite{nerz2013timeevolutionofCMC}, we find that the Euclidean coordinate centers of the CMC-leaves oscillate as the surfaces approach infinity for appropriate choices of $f$, see Section \ref{York}. In this setting, we use an \lq artificial\rq\ spacetime which probably has no physical significance; in particular, we do not specify the remaining data $(K,\rho,J)$ needed to represent an initial data set.

The second class of examples arises as time-slices in the Schwarzschild spacetime \eqref{schwarzschild_def}, in particular as bounded graphs over the canonical time-slice $\lbrace{t=0\rbrace}$. As the Schwarzschild spacetime satisfies Einstein's vacuum equation -- i.\,e.\ Equation \eqref{EE} with $\mathfrak{T}\equiv0$ and hence $\rho\equiv0$, $J\equiv0$ --, and thus also the dominant energy condition \eqref{DEC}, these examples do seem physically relevant. However, the time-slices we study may seem somewhat artificial and definitely do not correspond to a conventional family of observers near infinity, see Section \ref{schwarz}.

In Section \ref{NG}, we provide similar examples in the framework of Newtonian gravity. These demonstrate that the observed divergence effect is not a deficiency of the definitions of the ADM- and the coordinate CMC-centers of mass in GR but a natural effect that one should expect to occur for any notion of center of mass modeled upon Newtonian ideas. In particular, the described effect should occur for any (coordinate) notion of center of mass the Newtonian limit of which coincides with the Newtonian center of mass, see above.

As far as we know, these examples are the first examples of asymptotically Schwarzschildean Riemannian metrics with divergent ADM- and in particular the first examples with divergent coordinate CMC-center of mass. Moreover, we do not know of any other examples where a critical order deviation from a given asymptotically flat metric gives rise to divergence of the center of mass. Explicit examples of asymptotically flat Riemannian metrics with divergent ADM-center of mass have been constructed by Beig and O'Murchadha \cite{BM_examples} and extended by Huang \cite{Huang_examples}. These examples satisfy Einstein's vacuum equation. Their CoM diverges due to a term of order higher than the critical one considered here. Chen, Huang, Wang, and Yau recently announced related examples of initial data sets with diverging ADM-center of mass arising as graphical time-slices in the Schwarzschild spacetime \cite{CHWY_ang}; however the time-slices considered there are not bounded and the induced Riemannian metrics are not asymptotically Schwarzschildean and not of critical order in our sense.

\subsection{Systems with prescribed CoM}\label{Introprescribed}
The center of mass of a spherically symmetric isolated gravitating system lies at the coordinate origin $\vec{0}$ both in NG and for the Riemannian Schwarzschild metric \eqref{schwarzschild_def} in GR. However, by merely translating the asymptotic coordinates $\vec{x}$ by a fixed vector $\vec{z}$, i.\,e.\ by using new asymptotic coordinates $\vec{y}:=\vec{x}+\vec{z}$, one can shift the coordinate center of mass to any prescribed position $\vec{z}$. In NG, such a coordinate translation affects the Newtonian potential $U(\vec{x})=-\frac{m}{r}+b(\vec{x})$ with $b=\mathcal{O}_{2}(r^{-2})$ at the order $r^{-2}\approx s^{-2}$ with $s:=\vert\vec{y}\vert$ via
$$U(\vec{y})=-\frac{m}{s}-\frac{m\vec{z}\cdot\vec{y}}{s^{3}}+b(\vec{y}-\vec{z})+\mathcal{O}_{2}(s^{-3}),$$
while the matter density $\rho$ is not affected at the critical order $r^{-4}\approx s^{-4}$ (because the new term is harmonic). Similarly, in the relativistic setting, a coordinate translation changes an asymptotically Schwarzschildean Riemannian metric $\g$ with asymptotic coordinates $\vec{x}$, $$\g_{ij}=:\g[m]_{ij}+\devi[]_{ij},\quad\devi[]_{ij}=\mathcal{O}_{2}(r^{-2}),$$ to another (but isometric) asymptotically Riemannian metric $$\g_{ij}=\g[m]_{ij}+\frac{2m\vec{z}\cdot\vec{y}}{s^{3}}\,\delta_{ij}+\devi[]_{ij}+\mathcal{O}_{2}(s^{-3})$$ with respect to the translated coordinates $\vec{y}$, $s:=\vert\vec{y}\vert$, and the Euclidean metric $\delta$. The change occurs at the critical order $r^{-2}\approx s^{-2}$. The scalar curvature is not affected at the critical order $r^{-4}\approx s^{-4}$ (as the contribution of the new term cancels).

For given $m>0$ and $\vec{z}\in\R^{3}$, one can reinterpret the above consideration by saying that the asymptotically Schwarzschildean Riemannian manifold $(M^{3},\g,\vec{x}\,)$ defined by $$\g_{ij}:=\g[m]_{ij}+\frac{2m\vec{z}\cdot\vec{x}}{r^{3}}\,\delta_{ij}$$ is a metric of prescribed mass $m$ and ADM- and coordinate CMC-center of mass $\vec{z}$. This metric deviates from the Riemannian Schwarzschild metric at the critical order $r^{-2}$ but the change has no effect at the corresponding critical order $r^{-4}$ of the scalar curvature, see \cite[Sec.3.3]{Cederbaum_newtonian_limit}. We will construct asymptotically Schwarzschildean Riemannian metrics with prescribed mass $m>0$ and ADM- and coordinate CMC-center of mass $\vec{z}$ such that the deviation from the Schwarzschild metric and the scalar curvature of the metrics decay at the critical rates $r^{-2}$ and $r^{-4}$, respectively, see Section \ref{prescribed}. Again, we construct a first example by modifying the Riemannian Schwarzschild metric by a certain multiple of a suitably chosen York tensor. A second example arises as a suitably chosen bounded time-slice of the Schwarzschild spacetime. We present a similar construction in the Newtonian setting, underlining that this is a natural effect to be expected of isolated gravitating systems.

Corvino \cite{Corvino}, Corvino and Schoen \cite{corvino2006asymptotics}, and Huang, Schoen and Wang \cite[Cor.~5.3]{huang2011specifying} have also constructed suitably asymptotically flat metrics with prescribed mass and center of mass. Their metrics arise as essentially conformally flat solutions to the vacuum Einstein constraint equations. Recently, Chen, Huang, Wang, and Yau announced examples of initial data sets with prescribed ADM-center of mass which arise as unbounded not asymptotically Schwarz\-schil\-dean graphical time-slices in the Schwarzschild spacetime \cite{CHWY_ang}. None of these examples are of critical order in our sense.

In the description of the explicit examples as well as in the $\mathcal{O}$-notation introduced in Section \ref{notation}, we disrespect physical units for notational simplicity. To remedy this problem, one can replace the radial coordinate $r$ (or the foliation parameter $\sigma$) by $\nicefrac{r}{r_{0}}$ (or $\nicefrac{\sigma}{\sigma_{0}}$) for a fixed radius $r_{0}>0$ (or parameter $\sigma_{0}>0$, respectively).\\

\textbf{Structure of the paper.} In Section \ref{notation}, we introduce the relevant notations and define the asymptotic decay we will work with. In Section \ref{NG}, we quickly review the relevant facts and definitions from Newtonian gravity. We then discuss the delicate influence of the decay of the matter density $\rho$ of an isolated gravitating system on the convergence of the center of mass (volume integral). The discussed results are well-known; however, we provide explicit examples which will enable us to highlight the similarities to the relativistic situation. The first class of relativistic examples is presented in Section \ref{York}, the second class in Section \ref{schwarz}. Finally, we construct explicit examples of isolated gravitating systems of critical decay with prescribed mass $m>0$ and center of mass $\vec{z}$ in Sections \ref{subNG} (Newtonian example), \ref{subYork} (first relativistic example), and \ref{subschwarz} (second relativistic example). We summarize the examples and the different notions of mass and center of mass discussed in this paper in a table on page \pageref{table}. \\ 

\smallskip\textbf{Acknowledgment.}
The authors thank Gerhard Huisken for helpful discussions and Mu-Tao Wang for comments on the first draft of this paper.\smallskip
\pagebreak[3]

%% file: notation.tex
\section{Notation and decay conditions}\label{notation}
The contemporary literature knows several different types of decay conditions for isolated gravitating systems. We will use \emph{pointwise} assumptions. Others assume Sobolev regularity, i.\,e.\ they assume that $\g-\schwg$ and $\g-\eukg$ lie in specific weighted Sobolev spaces\footnote{see for example Huang, Schoen, and Wang \cite[Thm~1,\;Thm~2]{huang2011specifying}} or satisfy certain symmetry conditions such as the Regge-Teitelboim conditions\footnote{see for example Huang \cite{Lan_Hsuan_Huang__Foliations_by_Stable_Spheres_with_Constant_Mean_Curvature} or Eichmair and Metzger \cite{metzger_eichmair_2012_unique}}. We note that the decay assumptions satisfied by our examples are stronger. In particular, the corresponding Sobolev assumptions hold. We expect that our examples can be generalized to higher dimensions $n\ge3$.

We use the following $\Oof*$-notation corresponding to pointwise inequalities.

\begin{notation}
Let $(\M,\g)$ be a Riemannian manifold, $R_{0}>0$ a fixed radius, and $\vec x:\M\setminus L\to\R^3\setminus B_{R_{0}}(\vec 0)$ a chart on $\M$ outside a compact set $L\subseteq\M$. For $s\in\R$, $f,h\in\Ck^0(\M\setminus L)$, we write $f=h+\Oof[r]s$, if there is constant $C$ such that
\begin{align*} \vert f-h\vert \le{}& C\cdot r^s \end{align*}
in $\M\setminus L$, where $r(\vec{x}\,):=\vert \vec x\vert$ denotes the Euclidean distance to the coordinate origin $\vec{0}$. Similarly, we write $f=h+\Oof_k[r]s$ if $s\in\R$ and $f,h\in\Ck^k(\M\setminus L)$ satisfy
\begin{align*} \frac{\partial*^\alpha(f-h)}{\partial*x^\alpha} \in \Oof[r]{s-\vert\alpha\vert} \qquad\forall\,\vert\alpha\vert\le k. \end{align*}
We will use the same notation for tensor fields, replacing partial by covariant derivatives and absolute values by tensor norms induced by $\g$. Moreover, we will use this notation in the context of foliations near infinity, replacing the radial coordinate $r$ by the foliation parameter $\sigma$.
\end{notation}

Let us now quickly recall the Schwarzschild spacetime. The Schwarzschild spacetime is the oldest known non-trivial solution of Einstein's (vacuum) equation \eqref{EE} with $\mathfrak{T}\equiv0$. It describes the spacetime exterior region of a spherically symmetric star or a spherically symmetric black hole. In the context of isolated gravitating systems, one mostly considers the Riemannian Schwarzschild time-slice $\lbrace{t=0\rbrace}$.
\begin{definition}[Schwarzschild spacetime]\label{schwarzschild_spaccetime}
Let $m>0$. The \emph{Schwarzschild spacetime $(\DoATensor{\mathfrak M}[m],\DoATensor{\mathfrak g}[m])$ of mass $m$} is defined by $\DoATensor{\mathfrak M}[m]:=\R\times(\R^3\setminus B_{\frac m2}(\vec{0}))$ and
\begin{align*}
 \DoATensor[\hspace{-.05em}]{\mathfrak g}[m] := {-}(\frac{1-\frac m{2r}}{1+\frac m{2r}})^2\d t^2 + \schwg, \quad\text{with}\quad \schwg := (1+\frac m{2r})^4 \eukg, \labeleq{schwarzschild_def}
\end{align*}
where $t$ is the time-coordinate and $\eukg$ denotes the Euclidean metric on $\R^3$. The metric $\schwg$ is called the \emph{(Riemannian) Schwarzschild metric of mass $m$}. The function
\begin{align*}\labeleq{Nschwarz}
\schwN[m] :=&\frac{1-\frac{m}{2r}}{1+\frac{m}{2r}}
\end{align*}
is called the \emph{(Schwarzschild) lapse function}.
\end{definition}
It is well-known that $(\DoATensor{\mathfrak M}[m],\DoATensor[\hspace{-.05em}]{\mathfrak g}[m])$ solves Einstein's vacuum equation, that the ADM-mass of ${\schwg}$ is equal to the parameter $m$, and that the scalar curvature ${\sc[m]}$ of ${\schwg}$ is $0$ (and hence \emph{even} in any system of asymptotic coordinates, see page \pageref{antisymm}).\smallskip\pagebreak[3]

As there are different definitions of \lq asymptotically Schwarzschildean Riemannian manifolds\rq, we now precisely describe the asymptotic assumptions we make.

\begin{definition}[\texorpdfstring{$\Ck^k$}{C-k}-asymptotically Schwarzschildean Riemannian manifolds]\label{Ck_asymptotically_Schwarzschildean}
Let $\ve>0$. A triple $(\M,\g,\vec{x}\,)$ is called \term{$\Ck^k$-asymptotically Schwarz\-schildean of order $1+\ve$ with deviation $\erA$}, if $(\M,\g)$ is a smooth manifold, $\erA$ is a symmetric tensor field on $\M$, and $\vec x:\M\setminus L\to\R^3$ is a chart of $\M$ outside a compact set $L$ such that
\begin{align*}
 \g * ={}& \schwg + \erA + \Oof_k{-2-\ve} \text{ with } \erA ={} \Oof_k{-1-\ve}
\end{align*}
where $k\ge2$ and $r:\M\setminus L\to\interval0\infty:p\mapsto\vert \vec x(p)\vert$ is the Euclidean distance to the coordinate origin $\vec{x}=\vec{0}$.
\end{definition}\pagebreak[1]

%% file: Newtonian.tex
\section{The center of mass in Newtonian gravity}\label{NG}
Isolated gravitating systems are systems with matter distributions that decay suitably fast at infinity. In NG, the relevant decaying quantity is the matter density $\rho:\R^3\to\interval*0\infty$ of the system (at one instant of time). Alternatively, one can ask for suitable decay of the Newtonian potential $U$ solving the Poisson equation \eqref{NPoisson}. The total mass $m$ of the system is then given by the volume integral
\begin{align*}
 m :={}& \int_{\R^3}\rho\d V, \labeleq{Nmass}
\end{align*}
while its total center of mass $\centerz$ is defined as the (Euclidean) volume integral
\begin{align*}
 \centerz :={}& \frac1m\int_{\R^3}\rho\,\vec{x}\d V, \labeleq{NCoM}
\end{align*}
whenever the system is non-empty, i.\,e.\,$m>0$. Clearly, the mass and center of mass integrals converge only when $\rho$ decays suitably fast, e.\,g.\ when $\rho$ decays faster than $r^{-4}$. Alternatively, they converge when $\rho$ decays faster than $r^{-3}$ and is \emph{asymptotically even} in the sense that its odd part
\begin{align*}
\rho_{\text{odd}}(\vec{x}) :={}& \frac{\rho(\vec{x})-\rho({-}\vec{x})}2 \labeleq{antisymm}
\end{align*}
decays faster than $r^{-4}$ while its even part
\begin{align*}
\rho_{\text{even}}(\vec{x}) :={}& \frac{\rho(\vec{x})+\rho({-}\vec{x})}2 \labeleq{symm}
\end{align*}
need not. In terms of the $\Oof*$-notation introduced in Section \ref{notation}, this can be rephrased as follows.

\begin{proposition}\label{Nprop}
Let $\varepsilon>0$. Let $\rho:\R^3\to\interval*0\infty$ be a Newtonian matter density such that $\rho=\Oof[r]{-3-\varepsilon}$ as $r\to\infty$ and assume that its mass satisfies $m>0$. The corresponding center of mass is well-defined (in the sense of indefinite Riemann integrals in spherical polars) provided that
\begin{align*}
 \rho_{\text{odd}}=\Oof[r]{-4-\varepsilon}.
\end{align*}
\end{proposition}

\begin{proof}
When $\rho=\Oof[r]{-4-\varepsilon}$, the claim follows from Lebesgue's theorem on dominated convergence. Otherwise, we compute
\begin{align*}
\int_{B_{R}(0)}\rho\,\vec{x}\d V ={}& \int_{B_{R}(0)}\rho_{\text{odd}}\,\vec{x}\d V + \int_{B_{R}(0)}\rho_{\text{even}}\,\vec{x}\d V \\
 ={}& \int_{B_{R}(0)}\rho_{\text{odd}}\,\vec{x}\d V + \phantomas[c]{\int_{B_{R}(0)}\rho_{\text{even}}\,\vec{x}\d V}{\vec{0}}\hspace{-3.1em},
\end{align*}
for all $R>0$, where $B_{R}(0)$ denotes the open ball of radius $R$ centered at the origin. Using the assumption of asymptotic evenness, we find that the right hand side of this converges as $R\to\infty$, again by Lebesgue's theorem. This implies that the center of mass is well-defined.
\end{proof}
In particular, non-empty isolated Newtonian systems with even matter density (i.\,e.\ $\rho_{\text{odd}}\equiv0$) or even purely radial matter density $\rho=\rho(r)$ have a well-defined center of mass under the weaker decay condition $\rho=\Oof[r]{-3-\varepsilon}$. This decay corresponds to the relativistic Regge-Teitelboim conditions, see page \pageref{notation}.
 
Our matter becomes more delicate when $\rho=\Oof[r]-4$: While the center of mass \emph{can} still converge in this case -- e.\,g.\ if $\rho$ is even or purely radial or has compact support --, it need not as the example
\begin{align*}\labeleq{Nex}
\rho_{\vec{u}}:\R^3\to\interval*0\infty:\vec{x}\mapsto \frac{1}{r^{4}}\,(\vert\vec{u}\vert+\frac{\vec{x}\cdot\vec{u}}r)
\end{align*}
demonstrates\footnote{In fact, this matter density is not well-defined at the coordinate origin. The described effect however persists when it is multiplied by a cut-off function cutting out a ball centered at the coordinate origin and equal to $1$ in a neighborhood of $\infty$.}. Here, $\vec{u}\in\R^3\setminus\{0\}$ is an arbitrary fixed vector and $\vec x\cdot\vec u$ denotes the Euclidean dot product. The sole purpose of the term $\vert\vec{u}\vert$ is to achieve non-negativity of $\rho$. This divergence effect is well-known in probability theory as the existence of probability distributions \lq without expected value\rq\ or \lq with undefined moments\rq, e.\,g.\ the (one-dimensional) Cauchy- or Lorentz-distribution.

To further the analogy between NG and GR, it is useful to rewrite the volume integrals defining mass and center of mass of a system with density $\rho$ in terms of surface integrals using the Newtonian potential $U$ via the Poisson equation \eqref{NPoisson}. Following \cite[Chap.~4]{Cederbaum_newtonian_limit},  we define the quasi-local Newtonian mass and center of mass of a closed orientable surface $\Sigma\hookrightarrow\R^3$ as

 \begin{align*}\labeleq{NQLmass}
 m_{N}(\Sigma) :={}& \frac{1}{4\pi}\int_{\Sigma}\frac{\partial* U}{\partial*\nu}\d A\quad\text{ and }\\ \labeleq{NQLCoM}
\centerz_{N}(\Sigma) :={}& \frac{1}{4\pi\,m_{N}(\Sigma)}\int_{\Sigma}\left(\frac{\partial* U}{\partial*\nu}\,\vec{x}-U\frac{\partial*\vec x}{\partial*\nu}\right)\d A, 
\end{align*}
respectively, where $\d A$ is the area measure induced by the Euclidean metric and $\nu$ is the (Euclidean) outer unit normal to $\Sigma$. In the definition of $\centerz_{N}(\Sigma)$, we implicitly assume that $m_{N}(\Sigma)\neq0$. These surface integral expressions are related to the volume integrals for $m$ and $\centerz$ via \eqref{NPoisson} by the divergence theorem and Green's formula, respectively. Again by the divergence theorem, \eqref{NPoisson}, and $\rho\geq0$, one sees that $m_{N}(\Sigma)\geq0$ for all $\Sigma$. Whenever $\rho$ decays as in Proposition \ref{Nprop}, $m_{N}(\sphere^{2}_{r})$ and $\centerz_{N}(\sphere^{2}_{r})$ converge to $m$ and $\centerz$ as $r\to\infty$, respectively. Otherwise, if $\rho$ decays at the critical rate $\mathcal{O}(r^{-4})$ or more slowly, the surface integral $\centerz_{N}(\sphere^{2}_{r})$ converges as $r\to\infty$ if and only if the corresponding volume integral $$\int_{B_{r}(\vec{0})}\rho\,\vec{x}\,\d V$$ converges as $r\to\infty$ -- in close analogy to the relativistic case, see Section \ref{intro}.

Finally, the standard asymptotic expansion \eqref{Uexp} of the Newtonian potential $U$ ensures that the level sets of $U$ foliate a neighborhood of infinity when $m\neq0$. Moreover, the leaves -- which are called \emph{equipotential surfaces} -- become asymptotically round spheres centered at $\vec{z}$, \cite[p.\ 48]{Cederbaum_newtonian_limit}. The Euclidean coordinate center of this foliation coincides with $\vec{z}$. We will a-historically call this foliation the \emph{abstract Newtonian center of mass} in analogy to the abstract CMC-center of mass discussed in the introduction and in Section \ref{general}. 

In Section \ref{subNG}, we construct a matter density $\rho=\mathcal{O}(r^{-4})$ with arbitrarily prescribed mass $m>0$ and center of mass $\vec{z}$, see also Figure \ref{table}. 
\pagebreak[3]

%% file: general.tex
\section{The center of mass in general relativity}\label{general}
In general relativity, the ADM-mass $m$ of an isolated gravitating system modeled as an asymptotically decaying initial data set $(M^{3},\g,K,\rho,J)$ or an asymptotically Schwarzschildean Riemannian manifold $(M^{3},\g)$ is defined as follows.
\begin{definition}[ADM-mass]
The \term{ADM-mass $m$} of a $\Ck^k$-asymptotically Schwarz\-schildean Riemannian manifold $(\M,\g*,\vec{x}\,)$ of order $1+\ve$ for any $\ve>0$ is defined by
\begin{align*}
 m_{\text{ADM}} :={}& \frac1{16\pi}\lim_{r\to\infty}\sum_{i,j=1}^3\ \int_{\sphere^2_r(\vec 0)}(\partial_i@{\g_{jj}} - \partial_i@{\g_{ij}})\,\frac{x_i}r \d A, \labeleq{mADM} 
\end{align*}
where $\d A$ denotes the measure induced on the coordinate sphere $\sphere^2_r(\vec{0})$ by the Euclidean metric $\eukg$.
\end{definition}
The precise asymptotic decay necessary for defining the ADM-mass was discussed by Bartnik \cite{Bartnik} and Chru\'sciel \cite{Chrus2}; the decay we assume here is by far more restrictive. We will abbreviate the ADM-mass of a system $(\M,\g,\vec{x}\,)$ by $m$ and will from now on assume that it is strictly positive.\pagebreak[1]

We will now introduce the essential notion of the \emph{Euclidean coordinate center} of a surface. In Euclidean geometry, any closed surface $\Mz\hookrightarrow\R^3$ has a \lq Euclidean coordinate center\rq\ $\centerz_E(\Mz)$ defined by
\begin{align*} \centerz_i :=(\centerz_{E}(\Sigma))_{i}:={}& \fint_{\Mz} x_i \d A := \frac1{\vert\Mz\vert} \int_{\Mz} x_i \d A, \labeleq{eucl_center_of_surface} \end{align*}
where $\d A$ denotes the measure induced on $\Mz$ by the Euclidean metric $\eukg$. Picking a fixed system of asymptotically Schwarzschildean coordinates $\vec x:\M\setminus L\to\R^3$ of $(\M,\g)$, this definition can be extended to closed surfaces $\Mz\hookrightarrow\M\setminus L$. We will also call these \lq Euclidean centers\rq\ and denote them by $\centerz_E(\Mz)$. Following Huisken and Yau \cite[Thm 4.2]{huisken_yau_foliation}, we define:

\begin{definition}[Coordinate center of a foliation]\label{folicoordCoM}
Let $\{\Mz<\sigma>\}_{\sigma\ge\sigma_0}$ be a foliation near infinity of a $\Ck^k$-asymptotically Schwarz\-schildean Riemannian manifold  $(\M,\g*,\vec{x}\,)$. Let $\centerz<\sigma>_E$ denote the Euclidean coordinate center of the leaf $\Mz<\sigma>$. The \term{coordinate center $\centerz$ of the foliation $\{\Mz_\sigma\}_{\sigma\ge\sigma_0}$} is given by $$\centerz:=\lim_{\sigma\to\infty}\centerz<\sigma>_{E}$$ in case the limit exists. 
\end{definition}

\begin{remark}
Along a foliation with $r=\sigma+\Oof[\sigma]-\ve$, we could in principle use adapted non-Euclidean coordinate centers instead of Euclidean ones, i.\,e.
\begin{align*} \frac1{\Ag<\sigma>(\Mz<\sigma>)} \int_{\Mz<\sigma>} \vec x\,\d\,\Ag<\sigma> = \centerz<\sigma>_E + \Oof[\sigma]-\ve \qquad\text{if}\quad r=\sigma+\Oof[\sigma]-\ve,  \end{align*}
where $\Ag<\sigma>$ denotes the measure on $\Mz<\sigma>$ induced by the metric $\g$. As this does not make a difference, we will stick to the easier concept of Euclidean centers.
\end{remark}

The first geometric foliation used for the definition of a center of mass in general relativity is the foliation of $(\M,\g)$ near infinity by stable spheres with constant mean curvature introduced by Huisken and Yau. Existence and uniqueness of this so-called \term{CMC-foliation} was first proved by Huisken and Yau \cite{huisken_yau_foliation} in dimension three for $\mathcal C^4$-asymptotically Schwarzschildean Riemannian manifolds of order $1+\ve=2$. Subsequently, Metzger \cite{metzger2007foliations}, Huang \cite{Lan_Hsuan_Huang__Foliations_by_Stable_Spheres_with_Constant_Mean_Curvature} and Eichmair and Metzger \cite{metzger_eichmair_2012_unique} weakened the decay and regularity assumptions on the metric. In particular, such a CMC-foliation exists and is unique for any $\mathcal C^2$-asymptotically Schwarzschildean Riemannian manifold of order $1+\ve>1$.

Whenever the coordinate center of the CMC-foliation of $(\M,\g)$ exists in an asymptotic coordinate chart $\vec{x}$, we call it the \term{coordinate CMC-center of mass of $(\M,\g,\vec{x}\,)$} and denote it by $\vec{z}_{\text{CMC}}$.

The ADM-center of mass is defined as follows \cite{arnowitt1961coordinate}, \cite{BM_examples}.\pagebreak

\begin{definition}[ADM-center of mass]\label{ADMCoM}
The \term{ADM-center of mass} $\centerz_{\text{ADM}}\in\R^3$ of a $\Ck^k$-asymptotically Schwarz\-schildean Riemannian manifold $(\M,\g,\vec{x}\,)$ of order $1+\ve>0$ is defined by
\begin{align*}
 (\centerz_{\text{ADM}})_i :={}& \frac1{16\pi m} \lim_{r\to\infty}
\sum_{j=1}^3\,\int_{\sphere^2_r(0)}( (\partial_j@{\g_{jk}}-\partial_k@{\g_{jj}})\frac{x^{k}}{r}x_i- (\g_{ij}\,\frac{x^j}r - \g_{jj}\,\frac{x_i}r))\d A
\end{align*}
(whenever the limit exists), where all indices are raised and lowered with respect to the Euclidean metric $\eukg$.
\end{definition}

It is well-known that the coordinate CMC-center of mass  of a given asymptotically flat Riemannian metric $(\M,\g,\vec{x}\,)$ converges and coincides with the ADM-center of mass if the Regge-Teitelboim conditions hold, so that, in particular, the scalar curvature is asymptotically even, see Huang \cite{Lan_Hsuan_Huang__Foliations_by_Stable_Spheres_with_Constant_Mean_Curvature} and Eichmair and Metzger \cite{metzger_eichmair_2012_unique}. As proven by the second author \cite[Cor.~4.2]{nerz2013timeevolutionofCMC}, the same is true in the setting of a $\Ck^2$-asymptotically Schwarzschildean Riemannian manifold $(\M,\g,\vec{x}\,)$ -- without symmetry assumptions on the scalar curvature --, if the coordinate CMC- \emph{or} the ADM-center of mass converges. Whenever one and thus both of them are well-defined, they coincide.
\pagebreak[3]

%% file: York_Ex.tex
\section{Divergent examples related to motion}\label{York}
In this section, we construct examples of asymptotically Schwarzschildean Riemannian manifolds $(\M,\g,\vec{x}\,)$ for which the coordinate CMC-center of mass is not well-defined, i.\,e.\ the Euclidean centers of the leaves $\Mz<\sigma>$, $\centerz<\sigma>_E$, do not converge along the CMC-foliation. Therefore, the ADM-center of mass is not well-defined either \cite[Cor.~4.2]{nerz2013timeevolutionofCMC}. These examples are constructed in the following way: We begin with the Riemannian Schwarzschild manifold $(\R^3\setminus\{0\},\schwg)$ of mass $m>0$, see Definition \ref{schwarzschild_spaccetime}, perturb the metric by an arbitrary symmetric tensor field $\erA$ and calculate the change of the constant mean curvature surfaces under this perturbation. Then, we choose explicit tensors $\erA$ such that the Euclidean coordinate centers of the CMC-surfaces do not converge in $\R^3$, i.\,e.\ such that the coordinate CMC-center of mass is not well-defined.

Let $\overline\erA=\Oof_2{-1-\ve}$ be a symmetric tensor field on $\M:=\R^3\setminus B_R(\vec 0)$ for some $\ve>\nicefrac12$ and define interpolating Riemannian metrics $\g[t]:=\schwg-2t\,\overline\erA$, such that the deviations $\erAt$ of the constructed metrics are given by $\erAt={-}2t\,\overline\erA$. By \cite[Thm~6.4]{metzger2007foliations}, there exists a foliation $\{\Mz<\sigma>[t]\}_\sigma$ of $\M[3]$ by spheres $\Mz<\sigma>[t]$ with constant mean curvature $\nicefrac{-2}\sigma+\nicefrac{4m}{\sigma^2}$ with respect to $\g[t]$. Let $\centerz<\sigma>[t]_E\in\R^3$ denote the Euclidean co\-or\-din\-ate center of $\Mz<\sigma>[t]$, see \eqref{eucl_center_of_surface}.


It was shown by the second author \cite[Thm~3.1]{nerz2013timeevolutionofCMC} that these Euclidean coordinate centers evolve in the following way in $t$-direction:
\begin{corollary}[{\cite[Corollary of Thm~3.1]{nerz2013timeevolutionofCMC}}]
For the outer unit normal $\normal$ of $\Mz<\sigma>[t]$ with respect to the Euclidean metric, the asymptotic identity
\begin{align*}
 \partial[t]@{\centerz<\sigma>[t]_i} = \frac1{8\pi m}\int_{\Mz<\sigma>[t]} (\PiT_{ij}\normal^j + \sigma\;\delta_{ij}\PiT_{kl}^{,k}\normal^j\normal^l) \d A + \Oof[\sigma]-\ve \labeleq{York_Ex_Integral_1}
\end{align*}
holds, where $\PiT:=\tr\overline\erA\;\delta - \overline\erA$ and $\tr$ denotes the trace with respect to the Euclidean metric $\delta$.
\end{corollary}

\begin{remark}\label{ADMmom}
We recall that the \emph{ADM-linear momentum} ${\linmomentum}\in\R^3$ of a suitably decaying initial data set $(\M,\g,\k,\rho,J)$ is defined by
\begin{align*}
 (\linmomentum)_i = \frac1{8\pi m} \lim_{\sigma\to\infty}\int_{\sphere^2_\sigma(\vec{0})} \PiT_{ij}\frac{x^j}r \d A,
\end{align*}
where $\PiT=\tr\k\,\g-\k$ is the \term{momentum tensor field} of the initial data set \cite{arnowitt1961coordinate}.

Thus, the first term in the integral in \eqref{York_Ex_Integral_1} has to be understood as an approximate linear momentum, while the second term is an error term due to the approximation. This is explained in more detail in \cite[Remark~3.2]{nerz2013timeevolutionofCMC}.
\end{remark}

Furthermore, by \cite[Thm~1.1]{metzger2007foliations} and De Lellis and M\"uller \cite{DeLellisMueller_OptimalRigidityEstimates}, we know that the CMC-surface $\Mz<\sigma>[t]$ is a graph of a function ${\DoATensor\phi<\sigma>[t]}\in\Hk^1(\sphere^2_\sigma(\centerz<\sigma>[t]))$ over $\sphere^2_\sigma(\centerz<\sigma>[t])$ with $\Vert\phi\Vert_{\Hk^1(\sphere^2_\sigma(\centerz<\sigma>[t]))}=\Oof-\ve$, where we dropped the index $E$ for notational convenience. In particular, we find
\begin{align*}
 \partial[t]@{\centerz<\sigma>[t]_i} = \frac1{8\pi m}\int_{\sphere^2_\sigma(\centerz<\sigma>[t])} (\PiT_{ij}\DoATensor\eta<\sigma>[t]^j + \sigma\;\delta_{ij}\PiT_{kl}^{,k}\DoATensor\eta<\sigma>[t]^j\DoATensor\eta<\sigma>[t]^l) \d A + \Oof[\sigma]-\ve, \labeleq{York_Ex_Integral_2}
\end{align*}
where $\DoATensor\eta<\sigma>[t](\vec{x}\,)=\nicefrac{(\vec x-\centerz<\sigma>[t])}{\vert \vec x-\centerz<\sigma>[t]\vert}$ denotes the outer unit normal to $\sphere^2_\sigma(\centerz<\sigma>[t])$ with respect to the Euclidean metric. Moreover, we know that $\centerz<\sigma>[0]=\vec{0}$, since $\g[t=0]=\schwg$ is the Riemannian Schwarzschild metric, so that in particular $\vert\centerz<\sigma>[t]_i\vert=\Oof[\sigma]{1-\ve}$ by assumption on $\overline\erA=\Oof_2-\ve$ and by \eqref{York_Ex_Integral_2}. Thus, we conclude
\begin{align*}
 \partial[t]@{\centerz<\sigma>[t]_i}
	={}& \frac1{8\pi m} \int_{\sphere^2_\sigma(\vec{0})}
				(\PiT_{ij}\frac{x^j}r + \sigma\;\delta_{ij}\PiT_{kl}^{,k}\frac{x^j}r\frac{x^l}r) \d A	+ \Oof[\sigma]-\ve
\end{align*}
if we assume $\overline\erA=\Oof_2{-1-\ve}$. By a simple integration, the identity
\begin{align*}
 \centerz<\sigma>[t]_i ={}& \frac{t\sigma^2}{2m} \int_{\sphere^2_\sigma(\vec 0)}(\PiT_{ij}\frac{x^j}r + \sigma\;\delta_{ij}\PiT_{kl}^{,k}\frac{x^j}r\frac{x^l}r) \d A	+ \Oof[\sigma]-\ve \labeleq{York_Ex_Integral_Res}
\end{align*}
holds asymptotically for any $\overline\erA=\Oof_2{-1-\ve}$. Thus, the coordinate CMC-center of mass is well-defined if and only if \eqref{York_Ex_Integral_Res} converges as $\sigma\to\infty$. Let us now choose suitable tensor fields $\overline\erA=\Oof_2{-1-\ve}$ to get the desired examples. Motivated by York's example of a momentum tensor field  $\kY=\Oof_\infty-2$ with prescribed ADM-linear momentum $\vec P\in\R^3$ \cite[Chap.~6]{York__1979__Kinematics_and_dynamics_of_general_relativity}, we define $\overline{\erAt[]}:=f(r)\,\kY=\Oof_k{-1-\ve}$ for an arbitrary function $f=\Oof_k{1-\ve}$ with $k\ge2$. Recall that the York tensor is given by
\begin{align*}\labeleq{Yorktensor}
 \kY_{ij}(\vec{x}\,) :={}& \frac3{2r^2}\,(P_i\, \frac{x_j}{r} +  \frac{x_i}{r}\, P_j - P^k\frac{x_k}{r}\,(\eukg_{ij}-\frac{x_i\,x_j}{r^2})).
\end{align*}
Again, we raise and lower indices with respect to the Euclidean metric $\delta$. We find $\tr\,\overline{\erAt} = \Oof_k{-2-\ve}$ as well as
\begin{align*}
 \centerz<\sigma>[t]_E
	={}& \frac t2\,(2f(\sigma) + f'(\sigma)\;\sigma)\;\frac{\vec{P}}m + \Oof[\sigma]-\ve.
\end{align*}
Therefore, we conclude that
\begin{align*}\labeleq{CoMt,f}
 \centerz[t] :={}& \lim_{\sigma\to\infty}\centerz<\sigma>[t] = \frac{t \vec{P}}{2m}\,\lim_{\sigma\to\infty} \big(2f(\sigma) + f'(\sigma)\;\sigma\big),
\end{align*}
whenever this limit exists.

In particular, by choosing $t=1$ as well as the specific functions $f(r):=\sin(\ln r)$ and $f(r):=r^{1-\ve}$ with $\ve\in\interval{\nicefrac12}1$, we get examples of $\Ck^\infty$-asymptotically Schwarz\-schildean manifolds $(\R^3\setminus B_R(\vec{0}),\g,\vec{x}\,)$ of order $2$ and $1+\ve$, respectively, for which the coordinate CMC-center of mass (and therefore the ADM-center of mass) do not exist. For $f(r)=\sin(\ln r)$, the Euclidean centers of the leaves $\Mz<\sigma>$ of the CMC-foliation oscillate as $\sigma\to\infty$. This conflicts with \cite[Thm~4.2]{huisken_yau_foliation}; the first inequality on p.~301 seems unjustified. However, they are bounded as asserted by \cite[Proof~Thm~5.1]{huisken_yau_foliation}. When $f(r)=r^{1-\ve}$, the Euclidean centers of the leaves $\Mz<\sigma>$ of the CMC-foliation are unbounded; but we note that the decay conditions of \cite{huisken_yau_foliation} are violated.

Let us now summarize the above discussion, see also Figure \ref{table}.
\begin{example}[Divergent CoM $1$]\label{Ex1}
Let $m,\ve>0$ and let $Y$ be the York-tensor defined in \eqref{Yorktensor} for some $\vec{0}\neq\vec{P}\in\R^{3}$. Let
\begin{align*}
f:\left(\frac{m}{2},\infty\right)\to\R:r\mapsto\sin(\ln r)\quad\text{or}\quad
f:\left(\frac{m}{2},\infty\right)\to\R:r\mapsto r^{1-\ve}.
\end{align*}
Set $\g[f]:=\g[m]-2f(r)\kY$ so that $\g[f]=\g[m]+\mathcal{O}_{\infty}(r^{{-2}})$ or $\g[f]=\g[m]+\mathcal{O}_{\infty}(r^{-1-\ve})$, respectively, on $M^{3}:=\R^{3}\setminus B_{{m}/{2}}(\vec{0})$. Neither the coordinate CMC-center of mass nor the ADM-center of mass of the Riemannian manifolds $(M^{3},\g[f])$ are well-defined. The first example decays at the critical order $1+\ve=2$ as discussed in Section \ref{intro}.
\end{example}

A similar construction allows us to identify $C^{\infty}$-asymptotically Schwarz\-schil\-dean Riemannian metrics of the critical order $1+\ve=2$ with arbitrarily prescribed mass $m>0$ and ADM- and coordinate CMC-center of mass $\vec{z}$, see Section \ref{subYork} and Fig.~\ref{table}.

%% file: Schwarzschild.tex
\section{Divergent examples in the Schwarzschild spacetime}\label{schwarz}
The second class of examples of asymptotically Schwarzschildean Riemannian manifolds or initial data sets with divergent ADM- and coordinate CMC-center of mass arises as time-slices of the Schwarzschild spacetime $(\DoATensor{\mathfrak M}[m],\DoATensor{\mathfrak g}[m])$, see \ref{schwarzschild_spaccetime}. As the effects we exploit have nothing to do with spherical symmetry or staticity, similar examples can be constructed in many spacetimes (e.\,g.\ in any time-symmetric one ($K\equiv0$) with compactly supported or suitably decaying matter distribution).

Clearly, the center of mass of the \emph{canonical} time-slice $\lbrace{t=0\rbrace}$ of the Schwarz\-schild spacetime is the coordinate origin, $\vec{z}_{\text{ADM}}=\vec{z}_{\text{CMC}}=\vec{0}$. Other time-slices of the Schwarz\-schild spacetime can serve as a guideline for gaining a better understanding of how different families of observers near infinity perceive the spacetime and its center of mass, in particular.

Here, we will focus on what we call \emph{graphical time-slices} $\slice$ arising as graphs of smooth functions $T:\R^{3}\setminus B_{\nicefrac{m}{2}}(\vec{0})\to\R$ \lq over\rq\ the canonical time-slice (in time-direction), i.\,e.\ 
\begin{align*}\labeleq{ST}
\slice:=\lbrace t=T(\vec{x})\,|\,\vec{x}\in \R^{3}\setminus B_{\frac{m}{2}}(\vec{0})\rbrace.
\end{align*}
To most easily comply with the asymptotic decay conditions specified in Section \ref{general}, we will assume that $T=\mathcal{O}_{k}(r^{0})$ as $r\to\infty$, with $k\geq3$. Let $y^{i}:=x^{i}\vert_{\slice}$ denote the induced coordinates on $\slice$. In those coordinates, the metric $\g[T]$ induced on $\slice$ reads
\begin{align*}\labeleq{Smetric}
\g[T]_{ij}:=\g[T](\partial*_{y^{i}},\partial*_{y^{j}})=\g[m](\partial*_{x^{i}},\partial*_{x^{j}})-\schwN[m]^{2}\,T_{,i}\,T_{,j}=:\g[m]_{ij}-\schwN[m]^{2}\,T_{,i}\,T_{,j}
\end{align*}
with $\schwN[m]$ is as in Definition \ref{schwarzschild_spaccetime}. Here and in the following, we will only use partial derivatives with respect to the (spacetime) coordinates $x^{i}$; partial derivatives with respect to the coordinates $y^{i}$ adapted to the time-slice $\slice$ will be expressed in terms of partial derivatives with respect to $x^{i}$ via
\begin{align*}
\partial*_{y^{i}}=\partial*_{x^{i}}+T_{,i}\,\partial*_{t}.
\end{align*}
A straightforward computation shows that $(\R^{3}\setminus B_{\nicefrac{m}{2}}(\vec{0}),\g[T],\vec{y}\,)$ is $C^{k-1}$-asymptotically Schwarzschildean of order $1+\ve=2$ with deviation 
\begin{align*} \devi_{ij}:= {-}T_{,i}\,T_{,j}.\end{align*}
 Moreover, the ADM-mass of $(\R^{3}\setminus B_{\nicefrac{m}{2}}(\vec{0}),\g[T],\vec{y}\,)$ is $m$. To reduce notational complexity, we will drop the left upper index $m$ and understand all covariant derivatives and norms to be with respect to the metric $\g[m]$ unless stated otherwise. The future pointing unit normal $\Tnormal$ of $(\slice,\g[T])\hookrightarrow(\DoATensor{\mathfrak M}[m],\DoATensor{\mathfrak g}[m])$ is
 \begin{align*}
 \Tnormal=\frac{\partial*_{t}+N^{2}\,\nabla T}{N\sqrt{1-N^{2}\,\vert\nabla T\vert^{2}}},
 \end{align*}
 where $\nabla$ and $\vert\cdot\vert$ denote the covariant derivative and the norm induced on $\R^{3}\setminus B_{\nicefrac{m}{2}}(\vec{0})$ by the metric $\g[m]$, respectively. Accordingly, the second fundamental form of $(\slice,\g[T])\hookrightarrow(\DoATensor{\mathfrak M}[m],\DoATensor{\mathfrak g}[m])$ is given by
 \begin{align*}
 \K_{ij}=\frac{T_{,i}N_{,j}+N_{,i}T_{,j}+N \,\nabla^{2}_{ij}T-N^{3}\,\vert\nabla T\vert^{2}\,T_{,i}T_{,j}}{\sqrt{1-N^{2}\,\vert\nabla T\vert^{2}}},
 \end{align*}
 where $\nabla^{2}T$ denotes the covariant Hessian of $T$ with respect to $\g[m]$. From this, one computes that the ADM-linear momentum $\linmomentum$ of $\slice$ vanishes (see Remark \ref{ADMmom}), so that the ADM-energy-momentum-$4$-vector of $(\slice,\g[T],\K)$ coincides with the one of the canonical time-slice. Moreover, the ADM-angular momentum of $(\slice,\g[T],\K)$ vanishes as will be discussed elsewhere. Also, the induced mean curvature $\Ht[T]=\tr[T]\,\K$ satisfies
 \begin{align*}
 \Ht[T]=\laplace[\delta]T+\mathcal{O}_{k-2}(r^{-3})
 \end{align*}
 which means that the maximal slicing condition $\Ht=0$ implies $\laplace[\delta]T=\mathcal{O}_{k-2}(r^{-3})$. From $T=\mathcal{O}_{k}(r^{0})$, we deduce that the maximal slicing condition implies $T=\mathcal{O}_{k}(r^{-1})$ by the faster fall-off trick of the first author \cite[Thm~1.4.10]{Cederbaum_newtonian_limit}.
 
When evaluating the ADM-center of mass surface integral from Definition \ref{ADMCoM} on a finite coordinate sphere with respect to the $\vec{y}$-coordinates in $\slice$, we find
\begin{align*}
(\centerz[T\!](\mathbb{S}^{2}_{r}(\vec{0})))_{l}={}&\frac{1}{16\pi m}\sum_{i,j=1}^{3}\,\int_{\mathbb{S}^{2}_{r}(\vec{0})}\left\{ (\devi[]_{ij,i}-\devi[]_{ii,j})\frac{x_{j}x_{l}}{r}-(\devi[]_{il}\frac{x_{i}}{r}-\devi[]_{ii}\frac{x_{l}}{r})\right.\\
&\left.\phantom{ (\devi[]_{ij,i}-\devi[]_{ii,j})\frac{x_{j}x_{l}}{r}}\quad+\mathcal{O}_{k-2}(r^{-3})\right\}\d A(\vec{x})\\\labeleq{ADMCoMT}
={}&\frac{1}{16\pi m}\int_{\mathbb{S}^{2}_{r}(\vec{0})}\lbrace (\div\devi[])(\nu)x_{l}-\nu(\tr\devi[])x_{l}-\devi[](\nu,\partial_{l})+\tr\devi[]\,\nu_{l}\rbrace\d A(\vec{x})\\
&+\mathcal{O}(r^{-1}),
\end{align*}
where $\nu$ and $\div$ are the Euclidean unit normal to $\mathbb{S}^{2}_{r}(\vec{0})$ and the Euclidean divergence, respectively. We dropped the left upper index $T$ on the deviation $\devi$ for notational convenience and because this formula in fact applies to any deviation $\devi[]=\mathcal{O}_{k}(r^{-2})$ with $k\geq2$, see also \eqref{York_Ex_Integral_2} where $\overline\erA=-2t\erA$. As discussed in Section \ref{intro}, Corvino and Wu \cite{Corvino-Wu} observed that the ADM-surface integrals correspond to volume integrals over $\sc\,\vec{x}$. Here, this can be explicitly seen by applying the Euclidean divergence theorem:
\begin{align*}
\centerz[T\!](\mathbb{S}^{2}_{r}(\vec{0}))={}&\centerz[T\!](\mathbb{S}^{2}_{R_{0}}(\vec{0}))+\frac{1}{16\pi m}\int_{B_{r}(\vec{0})\setminus B_{R_{0}}(\vec{0})}\lbrace \div(\div\devi[]-\nabla\tr\devi[])\rbrace\,\vec{x}\,\d V+\mathcal{O}(r^{-1})
\end{align*}
for any fixed $R_{0}>\nicefrac{m}{2}$. In terms of $T$, this reads
\begin{align*}\labeleq{laplaHess}
\centerz[T\!](\mathbb{S}^{2}_{r}(\vec{0}))={}&\centerz[T\!](\mathbb{S}^{2}_{R_{0}}(\vec{0}))+\frac{1}{16\pi m}\int_{B_{r}(\vec{0})\setminus B_{R_{0}}(\vec{0})} \lbrace (\laplace T)^{2}-\vert\nabla^{2}T\vert^{2}\rbrace\vec{x}\,\d V\\
&+\mathcal{O}(r^{-1}).
\end{align*}
In particular, if $T$ gives rise to maximal slicing and thus as discussed above satisfies $T=\mathcal{O}_{k}(r^{-1})$ with $k\geq3$, the ADM- and thus also the coordinate CMC-center of mass converges. It is hence impossible to construct examples of maximally sliced bounded graph time-slices with diverging center of mass.

To construct an example with diverging center of mass, we pick
\begin{align*}
T:\R^{3}\setminus B_{\frac{m}{2}}(\vec{0})\to\R:\vec{x}\mapsto\sin(\ln r)+\frac{\vec{u}\cdot\vec{x}}{r}=\mathcal{O}_{\infty}(r^{0})
\end{align*}
for a fixed vector $\vec{0}\neq\vec{u}\in\R^{3}$. One computes from \eqref{laplaHess} that
\begin{align*}
\centerz[T\!](\mathbb{S}^{2}_{r}(\vec{0}))={}&\centerz[T\!](\mathbb{S}^{2}_{R_{0}}(\vec{0}))-\frac{1}{3m}(\cos(\ln r)-\cos(\ln R_{0}))\vec{u}
\end{align*}
which diverges as $r\to\infty$. Hence, the ADM- and thus also the coordinate CMC-center diverge \cite[Cor.~4.2]{nerz2013timeevolutionofCMC}.\pagebreak[3]

Finally, let us summarize the above discussion, see also Figure \ref{table}.
\begin{example}[Divergent CoM $2$]\label{Ex2}
Let $m>0$ and let $$T:\R^{3}\setminus B_{\frac{m}{2}}(\vec{0})\to\R:\vec{x}\mapsto\sin(\ln r)+\frac{\vec{u}\cdot\vec{x}}{r}$$ for a fixed vector $\vec{0}\neq\vec{u}\in\R^{3}$. Let $\g[T]$ be the metric induced on $\slice=\lbrace t=T(\vec{x})\rbrace$ by the ambient Schwarzschild spacetime $(\DoATensor{\mathfrak M}[m],\DoATensor{\mathfrak g}[m])$. The Riemannian manifold $(\slice,\g[T],\vec{y}\,)$ is $C^{\infty}$-asymptotically Schwarzschildean of the critical order $1+\ve=2$ and neither its coordinate CMC- nor its ADM-center of mass of $(\slice,\g[T])$ are well-defined.
\end{example}

Similar considerations allow us to construct $C^{\infty}$-asymptotically Schwarzschildean Riemannian metrics of the critical order $1+\ve=2$ with arbitrarily prescribed mass $m>0$ and ADM- and coordinate CMC-centers of mass $\vec{z}$, see Section \ref{subschwarz} and Figure \ref{table}. 
\pagebreak[3]

%% file: prescribed.tex
\section{Isolated gravitating systems with prescribed center of mass}\label{prescribed}
In Sections \ref{NG}, \ref{York}, and \ref{schwarz}, we have seen that various notions of center of mass diverge for specific asymptotically spherically symmetric examples decaying at the critical rates $\rho=\mathcal{O}(r^{-4})$ in NG and $\sc=\mathcal{O}(r^{-4})$ in GR, respectively, with corresponding deviation from the spherically symmetric case of critical orders, respectively, namely $U=-\frac{m}{r}+\mathcal{O}_{2}(r^{-2})$ in NG and $\g=\g[m]+\mathcal{O}_{2}(r^{-2})$ in GR. In this section, we will construct similar asymptotically spherically symmetric examples of the same critical orders of decay which possess a prescribed mass $m>0$ and center of mass $\vec{z}$. For a discussion of previous results on prescribed centers of mass in GR, see Section \ref{Introprescribed}.

Again, we will treat the Newtonian case first in Section \ref{subNG}, followed by relativistic examples based on an understanding of the motion of the CMC-center of mass, see Section \ref{subYork}, and by one arising as a time-slice in the Schwarzschild spacetime, see Section \ref{subschwarz}. As above, all examples necessarily rely on a violation of asymptotic evenness in the spirit of the Regge-Teitelboim conditions.

\subsection{Prescribed center of mass example in Newtonian gravity}\label{subNG}
Our task is most delicate in NG, as negative powers of the radius $r$ are singular at the origin. We will thus work with a smooth, rotationally symmetric cut-off function $\psi:\R\to\left[0,1\right]$ which vanishes in $B_{1}(\vec{0})$ and is equal to one outside $B_{2}(\vec{0})$. Set 
\begin{align*}\labeleq{psi}
a:=\int_{1}^{2}\frac{\psi(r)}{r^{2}}\,\d r>0.
\end{align*}
Recall that mass and center of mass are defined by \eqref{Nmass} and \eqref{NCoM}, respectively, in the Newtonian context.
If $\vec{z}=\vec{0}$, we set 
\begin{align*}\labeleq{rhom0}
\rho_{m,\vec{0}}:\R^{3}\to\left[0,\infty\right):\vec{x}\mapsto\frac{m \psi(r)}{4\pi (a+\frac{1}{2}) r^{4}}.
\end{align*}
This matter density satisfies $\rho_{m,\vec{0}}=\mathcal{O}(r^{-4})$ and has mass $m$ and center of mass $\vec{0}$. Accordingly, if $\vec{z}\neq\vec{0}$, set
\begin{align*}\labeleq{rhomz}
\rho_{m,\vec{z}}:\R^{3}\to\left[0,\infty\right):\vec{x}\mapsto\rho_{m,\vec{0}}(\vec{x}-\vec{z}).
\end{align*}
Clearly, $\rho_{m,\vec{z}}$ has mass $m$ and center of mass $\vec{z}$ as desired. Also, $\rho_{m,\vec{z}}$ decays at the critical rate $\rho_{m,\vec{z}}=\mathcal{O}(r^{-4})$. Note that $\rho_{m,\vec{z}}$ is not asymptotically even.

\subsection{Prescribed center of mass example related to motion}\label{subYork}
In Section \ref{York}, we constructed examples of $C^{k}$-asymptotically Schwarzschildean Riemannian metrics $(\R^{3}\setminus B_{R}(\vec{0}),\g[f],\vec{x}\,)$ by adding a deviation $\erAt[f]=-2f(r)\,\kY$ to the Schwarzschild metric, where $\kY$ is the York tensor defined in \eqref{Yorktensor} for arbitrary $\vec{P}\in\R^{3}$, and weight function $f=\mathcal{O}_{k}(r^{1-\ve})$ with $k\geq2$ and $\ve>0$. To construct a metric of the same form with prescribed mass $m$ and center of mass $\vec{z}$, we pick $\vec{P}:=\vec{z}$ and choose the constant weight $f(r)\equiv m=\mathcal{O}_{\infty}(r^{1-\ve})$ for $\ve=1$. The Riemannian metric $$\g[m,\vec{z}]:=\g[m]-2\,m\,\kY$$ is clearly $C^{\infty}$-asymptotically Schwarzschildean of the critical order $1+\ve=2$ and has ADM-mass $m$. Moreover, its coordinate CMC- and thus ADM-center of mass can be computed from \eqref{CoMt,f} (with $t=1$) and equals $\vec{z}$ as desired. The decay order $1+\ve=2$ corresponds to a decay of the scalar curvature of the critical rate ${\sc[m,\vec{z}]}=\mathcal{O}(r^{-4})$. The Regge-Teitelboim conditions are not satisfied for ${\g[m,\vec{z}]}$.

\subsection{Prescribed center of mass example in the Schwarzschild spacetime}\label{subschwarz}
In Section \ref{schwarz}, we constructed examples of $C^{k}$-asymptotically Schwarzschildean Riemannian manifolds arising as graphical time-slices of the Schwarzschild spacetime with graph functions $T:\R^{3}\setminus B_{\nicefrac{m}{2}}(\vec{0})\to\R$ of order $T=\mathcal{O}_{l}(r^{0})$, $l\geq3$. To concoct a metric arising in the same manner with prescribed mass $m$ and center of mass $\vec{z}$, we pick the ambient Schwarzschild spacetime of mass $m$. Let
\begin{align*}\labeleq{lambda}
\lambdaex[m,\vec{0}]:=0\quad\text{ and }\quad\lambdaex:=(\frac{15m}{8\vert\vec{z}\vert^{2}})^{\nicefrac{1}{3}}\quad\text{for }\vec{z}\neq\vec{0}.
\end{align*}
We set
\begin{align*}
\Tex:\R^{3}\setminus B_{\frac{m}{2}}(\vec{0})\to\R:\vec{x}\mapsto\lambdaex\,\frac{\vec{x}\cdot\vec{z}}{r}+(\lambdaex\,\frac{\vec{x}\cdot\vec{z}}{r})^{2},
\end{align*}
so that $\Tex=\mathcal{O}_{\infty}(r^{0})$. The induced metric $\g[m,\vec{z}]:=\g[m]-\schwN^{2}(d(\Tex))^{2}$ is thus clearly $C^{\infty}$-asymptotically Schwarzschildean of the order $1+\ve=2$ which again corresponds to the critical rate of decay of the scalar curvature, $\sc[m,\vec{z}]=\mathcal{O}(r^{-4})$. Moreover, it is obvious that $\g[m,\vec{z}]$ has ADM-mass $m$.  If $\vec{z}=\vec{0}$, we recover the canonical time-slice or $\Tex[m,\vec{0}]\equiv0$ and thus verify that $\centerz[m,\vec{0}]_{\text{ADM}}=\vec{0}$. To see that ${\g[m,\vec{z}]}$ has ADM- and thus coordinate CMC-center of mass $\vec{z}$ in case $\vec{z}\neq\vec{0}$, we use \eqref{ADMCoMT} applied to $\devi[T]=-dT^{2}$ with $T=\Tex$ and obtain
\begin{align*}
\centerz[m,\vec{z}]_{\text{ADM}}={}&\frac{1}{16\pi m}\lim_{r\to\infty}\sum_{i=1}^{3}\,\int_{\mathbb{S}^{2}_{r}(\vec{0})}(T_{,i}-T_{,ij}\,x^{j})\,T_{,i}\,\frac{\vec{x}}{r}\,\d A_{r}(\vec{x})\\
={}&\frac{\lambda^{3}}{\pi m}\int_{\mathbb{S}^{2}_{1}(\vec{0})}(\text{even}+\vert\vec{z}\vert^{2}\,\vec{z}\cdot\vec{\eta}-(\vec{z}\cdot\vec{\eta})^{3})\,\vec{\eta}\,\d A(\vec{\eta})\\
={}&\frac{\lambda^{3}}{\pi m}\int_{\mathbb{S}^{2}_{1}(\vec{0})}(\vert\vec{z}\vert^{2}\,\vec{z}\cdot\vec{\eta}-(\vec{z}\cdot\vec{\eta})^{3})\,\vec{\eta}\,\d A(\vec{\eta}),
\end{align*}
where we have dropped the upper left indices on $\lambda$ and $\Tex[]$ for notational convenience. The \lq even\rq\ term is even in the sense that its odd part defined as in \eqref{antisymm} vanishes and does thus not contribute to the integral by symmetry considerations. We compute
\begin{align*}
\centerz[m,\vec{z}]_{\text{ADM}}=\frac{\lambda^{3}}{\pi m}(\frac{4\pi}{3}\vert\vec{z}\vert^{2}\vec{z}-\frac{4\pi}{5}\vert\vec{z}\vert^{2}\vec{z})=\frac{8\lambda^{3}}{15m}\vert\vec{z}\vert^{2}\,\vec{z}=\vec{z},
\end{align*}
where we used well-known orthogonality properties of the spherical harmonics as well as the transformation formula applied to a special orthogonal transformation taking $\vec{z}$ to $\vert\vec{z}\vert\,\vec{e}_{1}$, for example, to simplify the computation. Note that the Regge-Teitelboim conditions are again not satisfied for $\g[m,\vec{z}]$.
\pagebreak[3]

%% file: table.tex
{\clearpage\vfill\small\begin{figure}[ht]
\begin{tabular}{c|c|c}
&&\\[-.6em]
mass and CoM expressions & NG, s.\ Sec.\ \ref{NG} & GR, s.\ Sec.\ \ref{general}\\[.5em]\hline\hline
&&\\[-.2em]
mass surface integrals & $m_{\text N}(\Sigma)$  & $m_{\text{ADM}}(\sphere^{2}_{r})$\\[.5em]\hline
&&\\[-.2em]
mass volume integrals & $\int\rho\,\d V$ & $\int \scalar\,\d V$\\[.5em]\hline
&&\\[-.2em]
total mass & $m$ & $m_{\text{ADM}}$\\[.5em]\hline\hline
&&\\[-.2em]
CoM surface integrals & $\centerz_{\text N}(\Sigma)$ & $\centerz_{\text{ADM}}(\sphere^{2}_{r})$, $\centerz_{E}(\Sigma_{\sigma})$\\[.5em]\hline
&&\\[-.2em]
CoM volume integrals & $\int\rho\,\vec{x}\,\d V$ & $\int \scalar\,\vec{x}\,\d V$\\[.5em]\hline
&&\\[-.2em]
total coordinate CoM & $\centerz$ & $\centerz_{\text{ADM}}=\centerz_{\text{CMC}}$\\[.5em]\hline
&&\\[-.2em]
abstract center of mass & level sets of $U$ & CMC-foliation \\[.5em]\hline\hline
&&\\[-.2em]
sufficient decay for & $\rho=\mathcal{O}(r^{-4-\varepsilon})$ & $\scalar=\mathcal{O}(r^{-4-\ve})$\\[.4em]
convergence of the CoM & $\Rightarrow U=-\frac{m}{r}+\mathcal{O}_{2}(r^{-2-\ve})$ & e.\,g.\ $\g_{ij}=\g[m]_{ij}+\mathcal{O}_{2}(r^{-2-\ve})$\\[.6em]\hline
&&\\[-.2em]
critical order & $\rho=\mathcal{O}(r^{-4})$ & $\scalar=\mathcal{O}(r^{-4})$\\[.3em]
of decay &  $\Rightarrow U=-\frac{m}{r}+\mathcal{O}_{2}(r^{-2})$ &  e.\,g.\ $\g_{ij}=\g[m]_{ij}+\mathcal{O}_{2}(r^{-2})$\\[.6em]\hline
&&\\[-.2em]
traditional decay & asymptotic evenness & Regge-Teitelboim \\
assumption & & conditions\\[.6em]\hline
&&\\[-.2em]
explicit examples & $\rho_{\vec{u}}(\vec{x})=\frac{1}{r^{4}}(\vert\vec{u}\vert+\frac{\vec{x}\cdot\vec{u}}r)$ & $\g=\g[m]-2\sin(\ln r)\,\kY$\\[.3em]
with diverging CoM & with $\vec{0}\neq\vec{u}\in\R^{3}$, Sec.\ \ref{NG}, & $\kY$ York tensor, Sec.\ \ref{York} \\[.6em]
& & and graph in Schwarz-\\[.2em]
& & schild spacetime: \\[.2em]
& & $t=T(\vec{x})=\sin(\ln r)+\frac{\vec{x}\cdot\vec{u}}{r}$\\[.3em]
& & with $\vec{0}\neq\vec{u}\in\R^{3}$, Sec.\ \ref{schwarz} \\[.6em]\hline
&&\\[-.2em]
explicit examples & $\rho_{m,\vec{z}} \propto \frac{m}{4\pi \vert\vec{x}-\vec{z}\vert^{4}}$ & $\g[m,\vec{z}]=\g[m]-2\,m\,\kY$ \\[.3em]
with prescribed mass $m$ & cut-off at $\vec{z}$, Sec.\ \ref{subNG} & with $\kY$ York tensor \\[.2em]
and CoM $\vec{z}$ &  & for $\vec{P}=\vec{z}$, Sec.\ \ref{subYork} \\[.6em]
and of critical order & & and graph in Schwarz-\\[.2em]
of decay & & schild spacetime,  Sec.\ \ref{subschwarz}:\\[.2em]
& & $t=\Tex[](\vec{x})=\lambda\,\frac{\vec{x}\cdot\vec{z}}{r}+(\lambda\,\frac{\vec{x}\cdot\vec{z}}{r})^{2}
$ \\[.2em]
& & with $\lambdaex=(\nicefrac{15m}{8\vert\vec{z}\vert^{2}})^{\nicefrac{1}{3}}$\\[.6em]
\end{tabular}
\caption{Different expressions for mass and center of mass in Newtonian gravity (NG) and general relativity (GR) as well as decay considerations ($\ve>0)$. $\g[m]$ is the Riemannian Schwarzschild metric of mass $m>0$. \label{table}}
\end{figure}\vfill}